\documentclass[reqno]{amsart}

\usepackage{amssymb, amsmath, amsthm}

 \newtheorem{thm}{Theorem}[section]
  \newtheorem{cor}[thm]{Corollary}
 \newtheorem{lem}[thm]{Lemma}
 \newtheorem{prop}[thm]{Proposition}

    \theoremstyle{definition}
 \newtheorem{defn}[thm]{Definition}
 \newtheorem{defns}[thm]{Definitions}
  \newtheorem{notn}[thm]{Notation}
  \newtheorem{example}[thm]{Example}
  \newtheorem{examples}[thm]{Examples}
 \theoremstyle{remark}
 \newtheorem{rem}[thm]{Remark}
\numberwithin{equation}{section}

\DeclareMathOperator{\Pspec}{P.Spec}

\DeclareMathOperator{\trdeg}{tr.deg}
\DeclareMathOperator{\ham}{ham}
\DeclareMathOperator{\hgt}{ht}
\DeclareMathOperator{\Phgt}{Pht}
\DeclareMathOperator{\Jac}{Jac}
\DeclareMathOperator{\Pz}{PZ}
\DeclareMathOperator{\sgn}{sgn}

\newcommand{\C}{{\mathbb C}}

\begin{document}

\title[ ]{Poisson spectra in polynomial algebras}

\author{David A. Jordan}
\author{Sei-Qwon Oh}

\address{School of Mathematics and Statistics, University of Sheffield, Hicks Building, Sheffield S3 7RH, UK} \email{d.a.jordan@sheffield.ac.uk}

\address{Department of Mathematics, Chungnam National  University, 99 Daehak-ro,   Yuseong-gu, Daejeon 305-764, Korea}
 \email{sqoh@cnu.ac.kr}

\thanks{The second author is supported by National  Research Foundation of Korea, Grant 2009-0071102, and thanks the Korea Institute for Advanced Study for the warm hospitality during a part of the preparation of this paper.}



\date{}


\begin{abstract}
A significant class of Poisson brackets on the polynomial algebra $\C[x_1,x_2,\dots, x_n]$ is studied and, for this class of Poisson brackets,
the Poisson prime ideals and Poisson primitive ideals are determined. Moreover it is established that these Poisson algebras satisfy the Poisson Dixmier-Moeglin equivalence.
 \end{abstract}

\maketitle

In \cite{JoOh}, the authors analyzed Poisson brackets on the polynomial algebra $\C[x,y,z]$ in three indeterminates $x,y,z$, including a class of Poisson brackets determined by Jacobians. In particular, for an arbitrary rational function $s/t\in\C(x,y,z)$, they analyzed the prime and primitive Poisson ideals for the Poisson bracket such that, for $f,g\in \C[x,y,z]$,
   \begin{equation}
   \{f,g\}=t^2\Jac(f,g,s/t),\label{qm}
   \end{equation} where $\Jac$ denotes the Jacobian determinant.

The main purpose of this paper is to generalize the results in \cite{JoOh} to the general polynomial algebra $A:=\C[x_1, x_2,\dots, x_n]$, $n\geq3$, equipped with a Poisson bracket which is determined by $n-2$ rational functions and which generalizes \eqref{qm}. As in \cite{JoOh}, the results will be illustrated using particular examples.

Fix $s_1,t_1,\dots , s_{n-2},t_{n-2}\in A$ such that $s_i$ and $t_i\neq0$ are coprime for each $i=1,2,\dots, n-2$. In Section 1 it is shown that there is a Poisson bracket on the quotient field $B$ of $A$ such that, for all $f,g\in B$,
 \[\{f,g\}=(t_1\dots t_{n-2})^2\Jac(f,g,s_1/t_1,s_2/t_2,\dots,s_{n-2}/t_{n-2}).\]
The purpose of the factor $(t_1\dots t_{n-2})^2$ is to ensure that this restricts to a Poisson bracket on $A$.

The Poisson prime ideals of $A$ for the above bracket are determined in Section 2 where Definition~\ref{resnullproper} uses the terminology \emph{residually null}, respectively \emph{proper}, for Poisson prime ideals $P$ where the induced Poisson bracket on $A/P$ is zero, respectively non-zero.
 The residually null Poisson prime ideals of $A$ form a Zariski closed set of the prime spectrum of $A$ and can often be found explicitly using elementary commutative algebra. We shall determine the proper Poisson prime ideals of $A$  in terms of a finite set of localizations $A_\gamma$ of $A$, each of which has a subalgebra $C_\gamma$
that is a polynomial ring in $n-2$ variables and is contained in the Poisson centre of $A_\gamma$. As the Poisson bracket on $C_\gamma$ is trivial, any prime ideal $Q$ of $C_\gamma$ is Poisson. Although $QA_\gamma$ need not be prime, it is a Poisson ideal and the finitely many minimal prime ideals of $A_\gamma$ over $QA_\gamma$ are Poisson prime ideals of $A_\gamma$. Taking the intersection of each of these with $A$, we obtain finitely many Poisson prime ideals of $A$. The main result is that every proper Poisson prime ideal $P$ of $A_\gamma$ occurs in this way with $Q=PA_\gamma\cap C_\gamma$. The passage between Poisson prime ideals of $A_\gamma$ and those of $A$ can then be handled by standard localization techniques. This will be illustrated using examples with $n=4$ in which case the algebras $C_\gamma$ are polynomial algebras in two indeterminates. The main example is the Poisson bracket associated with $2\times 2$ quantum matrices with which the reader may be familiar.
We also consider actions on $A$, as Poisson automorphisms, of subgroups of the multiplicative group $(\C^*)^n$.

In Section 3, we determine the Poisson primitive ideals of $A$ and show that $A$ satisfies the Poisson Dixmier-Moeglin equivalence discussed in \cite[2.4]{Oh4} and \cite{Good3}. Here, as indeed is the case with the Poisson prime ideals, the varieties determined by  $n-2$ polynomials of the form $\lambda_is_i-\mu_it_i$, $i=1,2,\dots,n-2$, where $(\lambda_i,\mu_i)\in\C^2\backslash\{(0,0)\}$ for all $i$, play an important role.

\section{Poisson brackets}\label{Pbrackets}
\begin{defn}
A \emph{Poisson algebra} is $\C$-algebra $A$ with a Poisson bracket, that is a bilinear product
$\{-,-\}:A\times A\rightarrow A$ such that $A$ is a Lie algebra
under $\{-,-\}$ and, for all $a\in A$, the \emph{hamiltonian} $\ham(a):=\{a,-\}$ is a
$\C$-derivation of $A$.
\end{defn}

\begin{notn}
Let $A$ denote the polynomial algebra $\C[x_1,\dots,x_n]$ in $n$ indeterminates and let $B$ denote the quotient field $\C(x_1,\dots,x_n)$ of $A$.
For $1\leq i\leq n$, let $\partial_i$ be the derivation $\frac{\partial}{\partial x_i}$ of $B$.
For $b_1,b_2,\dots,b_n\in B$, let $\Jac_M(b_1,b_2,\dots,b_n)$ denote the Jacobian matrix
$(\partial_j(b_i))$ and let $\Jac(b_1,b_2,\dots,b_n)$ denote the Jacobian determinant $|\Jac_M(b_1,b_2,\dots,b_n)|$.
Thus the $i$th row of $\Jac_M(b_1,b_2,\dots,b_n)$ is $\nabla(b_i)$ where $\nabla=(\partial_1,\partial_2,\dots,\partial_n)$ is the gradient.

Let $a,f_1,f_2,f_3,\dots,f_{n-2}\in B$ and, for $f,g\in B$, let
\begin{equation}\{f,g\}=a\Jac(f,g,f_1,f_2,\dots,f_{n-2}).\label{definebracket}\end{equation}
Poisson brackets of this form, with $a=1$, appear in the literature of mathematical physics, for example see \cite{GMP,Przybysz}. Our aim in this section is to give an algebraic proof that \eqref{definebracket} defines a Poisson bracket on the rational function field $B$.

For an $(n-2)\times n$ matrix $M$ over $B$ and $1\leq i<j\leq n$, let $M_{ij}$ be the $(n-2)\times (n-2)$ minor obtained by deleting columns $i$ and $j$ of $M$ and taking the determinant.
Let $D$ be the $(n-2)\times n$ matrix with $i$th row $\nabla(f_i)$.  Then
\[\{x_i,x_j\}=(-1)^{i+j-1}aD_{ij}.\]
Also, if $u_1,u_2,\dots,u_{n-2}\in B$ are such that $a=u_1u_2\dots u_{n-2}$ then
\[\{x_i,x_j\}=(-1)^{i+j-1}E_{ij},\]
where $E$ is the $(n-2)\times n$ matrix with $i$th row $u_i\nabla(f_i)$.
\label{PBnotation}
\end{notn}

\begin{lem}\label{derivation}
Let $1\leq i\leq n$ and let $a,\phi_1, \phi_2,\dots,\phi_{i-1},\phi_{i+1},\dots \phi_n\in B$.
The map $\delta:B\rightarrow B$ given by
\[\delta(b)=a\Jac(\phi_1, \phi_2,\dots,\phi_{i-1},b, \phi_{i+1},\dots \phi_n)\]
is a derivation of $B$.
\end{lem}
\begin{proof}
Denoting the $ij$-minor of the Jacobian matrix $\Jac_M(\phi_1, \phi_2,\dots,\phi_n)$ by $m_{ij}$,
\[\delta=a((-1)^{i+1}m_{i1}\partial_1+(-1)^{i+2}m_{i2}\partial_2+\dots+(-1)^{i+n}m_{in}\partial_n).\]
Hence $\delta$ is a derivation of $B$.
\end{proof}

\begin{thm}\label{PB}
Let $a,f_1,f_2,f_3,\dots,f_{n-2}\in B$ and let $\{-,-\}$ be as in \eqref{definebracket}.
Then $\{-,-\}$ is a Poisson bracket on $B$.
\end{thm}
\begin{proof} Applying Lemma~\ref{derivation} with $(\phi_1, \phi_2,\dots,\phi_n)=(f,g,f_1,f_2,\dots,f_{n-2})$ and $i=1,2$, we see that, if $\{-,-\}$ is defined as in \eqref{definebracket},
then $\{f,-\}$ and $\{-,g\}$ are derivations. Also $\{-,-\}$ is clearly antisymmetric so it remains to show that it satisfies the Jacobi identity.

We begin with the case $a=1$ where we can exploit the $n$-Jacobi identity for the Jacobian \cite{panov}.
Given an ordered set $F=\{f_0,f_1,f_2,f_3,\dots,f_{n-2}\}$ of $n-1$ elements of $B$, define $\partial_F:B\rightarrow B$
\[
\partial_F(h)=\Jac(h,f_0,f_1,f_2,\dots,f_{n-2}).
\]
There is a minor difference here to \cite{panov} where $h$ appears in the rightmost argument.
Note that $\partial_F(f_i)=0$ for $0\leq i\leq n-2$.
When $a=1$,
$\{g,f\}=\partial_F(g)$,
where $F=\{f,f_1,f_2,f_3,\dots,f_{n-2}\}$.
The $n$-Jacobi identity for the Jacobian says that, for $h_1,h_2,\dots,h_n\in B$,
\begin{multline*}\partial_F(\Jac(h_1,h_2,\dots,h_{n-1},h_n))=\Jac(\partial_F(h_1),h_2,\dots,h_{n-1},h_n)\\+
\Jac(h_1,\partial_F(h_2),\dots,h_{n-1},h_n)+\dots+\Jac(h_1,h_2,\dots,\partial_F(h_n)).\end{multline*}
The proof of this in \cite{panov} is presented  for the algebra of $C^\infty$-functions on a real manifold but it is valid for the rational function field $B$.  It is first checked when each $h_i=x_i$ and then, using derivation properties as in the proof of \cite[Proposition 1.14]{JoOh}, extended
first to the polynomial algebra and then to the rational function field.

Let $f,g,h\in B$.
Then
\begin{align*}
&\{\{g,h\},f\}
=\partial_F(\{g,h\})
=\partial_F(\Jac(g,h,f_1,f_2,\dots,f_{n-2}))\\
=&\Jac(\partial_F(g),h,f_1,f_2,\dots,f_{n-2})
+\Jac(g,\partial_F(h),f_1,f_2,\dots,f_{n-2})\\
&\quad\quad \text{ (by the }n\text{-Jacobi identity, the other summands being }0\text{)}\\
=&\{\{g,f\},h\}+\{g,\{h,f\}\}=-\{\{f,g\},h\}-\{\{h,f\},g\}.
\end{align*}
Thus $\{-,-\}$ satisfies the Jacobi identity and is a Poisson bracket on $B$.

Now let $a\in B$. We need to show that the bracket $a\{-,-\}$ satisfies the Jacobi identity.
As \[a\{f,a\{g,h\}\}=a^2\{f,\{g,h\}\}+a\{g,h\}\{f,a\}\]
and $\{-,-\}$ satisfies the Jacobi identity, it suffices to show that
\[\{g,h\}\{f,a\}+\{f,g\}\{h,a\}+\{h,f\}\{g,a\}=0\]
for all $a,f,g,h\in B$. As
$\{g,h\}\{f,-\}+\{f,g\}\{h,-\}+\{h,f\}\{g,-\}$ and the similar maps, where three of $a,f,g,h$  are fixed, are derivations, it suffices to show that
\begin{equation}\{x_i,x_j\}\{x_k,x_\ell\}+\{x_k,x_i\}\{x_j,x_\ell\}+\{x_j,x_k\}\{x_i,x_\ell\}=0\label{ijkl}
\end{equation}
for $1\leq i\leq j\leq k\leq \ell\leq n$. Clearly \eqref{ijkl} holds when any two of $i, j, k,\ell$ are equal so we may assume that $i<j<k<\ell$. In this case \eqref{ijkl} is, using Notation~\ref{PBnotation},
\[D_{ij}D_{k\ell}-D_{ik}D_{j\ell}+D_{jk}D_{i\ell}=0.\]
This is a Pl\"{u}cker relation for the $(n-2)\times n$ matrix $D$, see
\cite[Theorem 1.3]{GKZ}, or \cite[Chapter VII \S6]{hodgepedoe}, where Pl\"{u}cker relations are called $p$-relations. Indeed it is one of the three-term Pl\"{u}cker relations stated explicitly in \cite[foot of p311]{hodgepedoe}. In the notation of \cite{hodgepedoe}, where subscripts indicate included rather than excluded rows, it is
\[
p_{i_1\dots i_{n-4}k\ell}p_{i_1\dots i_{n-4}ij}+p_{i_1\dots i_{n-4}\ell j}p_{i_1\dots i_{n-4}ki}
+p_{i_1\dots i_{n-4}\ell i}p_{i_1\dots i_{n-4}kj}=0,
\]
where $\{i_1,\dots,i_{n-4}\}=\{1,2,3,\dots,n\}\backslash\{i,j,k,\ell\}$.
\end{proof}

\begin{thm}\label{algdep}
If $f_1,f_2\dots, f_{n-2}$ are algebraically dependent over $\C$ then $\{-,-\}=0$.
\end{thm}
\begin{proof}
Let $0\neq G=G(y_1,\dots, y_{n-2})\in {\C}[y_1,\dots,y_{n-2}]$ be of minimal total degree such that $G(f_1,f_2,\dots, f_{n-2})=0$.
Without loss of generality, we may assume that the degree in $y_1$ of $G$ is at least one.
Let
\[G=\sum_{{\bf r}=(r_1,\dots, r_{n-2})}\alpha_{\bf r}y_1^{r_1}y_2^{r_2}\dots y_{n-2}^{r_{n-2}}.\]
Let $f,g\in B$ and let $\delta$ be the derivation in Lemma~\ref{derivation}, in the case where $i=3$, $\phi_1=f$, $\phi_2=g$ and, for $4\leq j\leq n$, $\phi_j=f_{j-2}$. Then $\delta(f_k)=0$ for $2\leq k\leq n-2$, whereas $\delta(f_1)=\{f,g\}$.
Then
\begin{align*}
0&=\delta(G(f_1,f_2,\dots, f_{n-2}))\\
&=\delta\left(\sum_{{\bf r}=(r_1,\dots, r_{n-2})}\alpha_{\bf r}f_1^{r_1}f_2^{r_2}\dots f_{n-2}^{r_{n-2}}\right)\\
&=\left(\sum_{\bf r}r_1\alpha_{\bf r}f_1^{r_1-1}f_2^{r_2}\dots f_{n-2}^{r_{n-2}}\right)\{f,g\}.
\end{align*}
By the minimality of $G$, $\sum_{\bf r}r_1\alpha_{\bf r}f_1^{r_1-1}f_2^{r_2}\dots f_{n-2}^{r_{n-2}}\neq 0$
so $\{f,g\}=0$.
\end{proof}

\section{Poisson spectra}
The following definitions and the claims made for them are well-known. Appropriate references include
\cite{goodsemiclass,Good3,JoOh}.

\begin{defns}\label{Pdefns}
Let $A$ be a Poisson algebra with bracket $\{-,-\}$. The \emph{Poisson centre} of $A$, denoted
$\Pz(A)$, of $A$ is $\{a\in A: \{a,r\}=0\mbox{ for
all }r\in A\}$.

An ideal $I$ of $A$ is a \emph{Poisson ideal} if $\{i,r\}\in I$ for all $i\in I$ and $r\in A$. If $I$ is a Poisson ideal of $A$
then $A/I$ is a Poisson algebra with $\{a+I,b+I\}=\{a,b\}+I$ for all $a, b\in A$. A Poisson ideal $P$ of $A$ is \emph{Poisson prime} if, for all Poisson ideals $I$ and $J$ of $A$,
$IJ\subseteq P$ implies $I\subseteq P$ or $J\subseteq P$. If $A$ is Noetherian then this is equivalent to saying that $P$ is both a prime ideal and a Poisson ideal. The \emph{Poisson spectrum} of $A$, written $\Pspec A$, is the set of all Poisson prime ideals of $A$. A maximal ideal $M$ of $A$ is said to be a \emph{Poisson maximal ideal} if it is also a Poisson ideal. This is not equivalent to saying that $M$ is maximal as a Poisson ideal.

The \emph{Poisson core} of an ideal $I$ of $A$, denoted $\mathcal{P}(I)$, is the largest Poisson ideal of $A$ contained in $I$. If $P$ is a prime ideal of $A$ then
$\mathcal{P}(I)$ is Poisson  prime.
A Poisson ideal $P$ of $A$ is \emph{Poisson primitive} if $P=\mathcal{P}(M)$ for some maximal ideal $M$ of $A$.
Every Poisson primitive ideal is Poisson prime.

If $\mathcal{S}$ is a multiplicatively closed subset of a Poisson algebra $A$ then the localization $A_\mathcal{S}$ is also a Poisson algebra, with $\{as^{-1},bt^{-1}\}$ computed using the quotient rule for derivations.
If $P$ is a Poisson prime ideal of $A$ then the quotient field $Q(A/P)$ is a Poisson algebra and $P$ is said to be \emph{rational} if $\Pz (Q(A/P))=\C$. For a Poisson prime ideal $P$ of an affine Poisson algebra $A$, there is, by \cite[1.7(i) and 1.10]{Oh4}, a sequence of implications:
$$P\text{ is locally closed $\Rightarrow P$ is Poisson primitive $\Rightarrow P$ is rational}.$$
To establish the \emph{Poisson Dixmier-Moeglin equivalence}, it is enough to show that if $P$ is a rational Poisson prime ideal of $A$ then $P$ is locally closed. For further discussion of this, see \cite{goodsemiclass,Good3}.

A $\C$-algebra automorphism $\theta$ of a Poisson algebra $A$ is a \emph{Poisson automorphism} of $A$ if $\theta(\{a,b\})=
\{\theta(a),\theta(b)\}$ for all $a,b\in A$ and is a \emph{Poisson anti-automorphism} of $A$ if $\theta(\{a,b\})=
\{\theta(b),\theta(a)\}$ for all $a,b\in A$. Under composition, the set of all Poisson automorphisms and Poisson anti-automorphisms of $A$ is a group in which the Poisson automorphisms form a normal subgroup of index $2$.

The height of a prime ideal $P$ of $A$ will be denoted $\hgt P$.
\end{defns}

\begin{defn}\label{resnullproper} Let $A$ be a Poisson algebra and $I$ be a Poisson ideal of $A$. Following \cite[Definition 1.8]{JoOh}, we say that
$I$ is \emph{residually
null} if the induced Poisson bracket on $A/I$ is zero.
 This is equivalent to saying that $I$ contains
all elements of
the form $\{a,b\}$ where $a,b\in A$, or that $I$ contains all such
elements where $a,b\in G$ for some generating set $G$ for $A$.
We shall also say that a Poisson ideal is a \emph{proper Poisson ideal} if it is not residually null.
\end{defn}

\begin{lem}\label{poissonprimitive}  Let $A$ be a Poisson algebra.
\begin{enumerate}
\item Every  residually null Poisson primitive ideal $P$ is a  Poisson maximal ideal.
\item A Poisson algebra $A$ is Poisson simple if and only if there does not exist a nonzero  Poisson primitive ideal of $A$.
 \item Let $A=\C[x_1,\dots,x_n]$ be a polynomial algebra with a Poisson bracket. Let $P$ be a proper Poisson prime ideal in $A$ of height $\geq n-2$. Then
$P$ is locally closed and Poisson primitive.

    \end{enumerate}
\end{lem}
\begin{proof} (1)
 Let $M$ be a maximal ideal of $A$ such that $P=\mathcal{P}(M)$. Suppose that $P$ is residually null. Then $M$ is Poisson and $P=\mathcal{P}(M)=M$ is maximal.

(2) The \lq only if\rq~ part is clear. For the converse, suppose that $A$ is not simple, let $I$ be a proper ideal of $A$ that is Poisson and let $M$ be a maximal ideal  of $A$ containing $I$. Then $\mathcal{P}(M)$ is Poisson primitive and $0\neq I\subseteq \mathcal{P}(M)$.

(3) Since $P$ is proper, $\{x_k,x_j\}\notin P$ for some  pair $k,j$. Let $Q$ be a Poisson prime ideal  containing $P$ properly. Then  $\hgt Q> n-2$ and hence  $Q$ is residually null by \cite[Proposition 3.2]{JoOh}.
It follows that $\{x_k,x_j\}\in Q$. Thus $P$ is locally closed and hence, by \cite[1.7(i)]{Oh4}, $P$ is Poisson primitive.
\end{proof}

\begin{notn} For the remainder of the paper, let $A=\C[x_1,\dots,x_n]$ and $B=\C(x_1,\dots,x_n)$, where $n\geq 3$. Let
$s_1,t_1,\dots, s_{n-2},t_{n-2}\in A$
be such that $s_i$ and $t_i\neq0$ are coprime for $1\leq i\leq n-2$. Let $f_i=s_1t_1^{-1}\in B$, $1\leq i\leq n-2$,
and let $a=t_1^2t_2^2\dots t_{n-2}^2$. By Theorem~\ref{PB}, there is a Poisson bracket on $B$ such that
$\{f,g\}=a\Jac(f,g,f_1,f_2,\dots,f_{n-2})$ for all $f,g\in B$.
Thus $\{f,g\}=\det J$, where $J$ is the $n\times n$ matrix with first row $\nabla f$, second row $\nabla g$ and, for $3\leq i\leq n$, $i$th row
$t_{i-2}^2\nabla(s_{i-2}t_{i-2}^{-1})$.
In the notation of \ref{PBnotation}, with each $u_i=t_i^2$,
\[\{x_i,x_j\}=(-1)^{i+j-1}aD_{ij}=(-1)^{i+j-1}E_{ij}.\]

Note that $t_i^2\partial_j(s_it_i^{-1})\in A$ for $1\leq i\leq n-2$ and
$1\leq j\leq n$. It follows that $\{f,g\}\in A$
for all $f,g\in A$ and hence that $A$ is a Poisson subalgebra of $B$.

 If $f_1,\dots, f_{n-2}$ are algebraically dependent over $\C$ then the Poisson bracket $\{-,-\}=0$, by Theorem~\ref{algdep}, so, henceforth, we shall assume that $f_1,\dots, f_{n-2}$ are algebraically independent over $\C$.
 \end{notn}

\begin{example}\label{qmat} 
Let $n=4$ and let \[s_1=x_1x_4-x_2x_3, \quad t_1=1, \quad s_2=x_2,\quad  t_2=x_3.\]
Then, in the notation of \ref{PBnotation}, \[E=\begin{pmatrix}t_1^2\nabla(f_1)\\ t_2^2\nabla(f_2)\end{pmatrix}=\begin{pmatrix}x_4&-x_3&-x_2&x_1\\ 0&x_3&-x_2&0\end{pmatrix}\]
and the resulting Poisson bracket on $\C[x_1, x_2, x_3, x_4]$ is such that:
\[\begin{aligned}
\{x_1,x_2\}&=x_1x_2, &\{x_1,x_3\}&=x_1x_3,&\{x_1,x_4\}&=2x_2x_3, \\
\{x_2,x_3\}&=0, &\{x_2,x_4\}&=x_2x_4,&\{x_3,x_4\}&=x_3x_4.
\end{aligned}
\]
This is the well-known Poisson bracket associated with $2\times2$ quantum
matrices, see \cite[2.9]{Oh4}. This example will be used to illustrate our methods and results.
\end{example}

\begin{example}\label{symm}
Let $n=4$ and let $s_1=x_1+x_2+x_3+x_4$, $s_2=x_1x_2+x_1x_3+x_1x_4+x_2x_3+x_2x_4+x_3x_4$ and $t_1=t_2=1$.
In the notation of \ref{PBnotation}, \[E=\begin{pmatrix}1&1&1&1\\x_2+x_3+x_4&x_1+x_3+x_4&x_1+x_2+x_4&x_1+x_2+x_3\end{pmatrix}\]
and, for the resulting Poisson bracket on $\C[x_1, x_2, x_3, x_4]$,
\[\begin{aligned}
\{x_1,x_2\}&=x_3-x_4, &\{x_1,x_3\}&=x_4-x_2,&\{x_1,x_4\}&=x_2-x_3, \\
\{x_2,x_3\}&=x_1-x_4, &\{x_2,x_4\}&=x_3-x_1,&\{x_3,x_4\}&=x_1-x_2.
\end{aligned}
\]
Here the elementary symmetric polynomials $s_1$ and $s_2$ are Poisson central. The prime ideal generated by $x_1-x_2$, $x_1-x_3$ and $x_1-x_4$ is residually null Poisson as are all the maximal ideals of the form $(x_1-\lambda,x_2-\lambda,x_3-\lambda,x_4-\lambda)$.

As $\{x_i,x_j\}$ is homogeneous of degree one, the Poisson bracket here is the Kirillov-Kostant-Souriau bracket,
\cite[III.5.5]{BGl}, for a $4$-dimensional Lie algebra $\mathfrak{g}$ in which $z:=x_1+x_2+x_3+x_4$ is central.
If $\mathfrak{s}=\mathfrak{g}/\C z$ then it is a routine calculation to  check that $[\mathfrak{s},\mathfrak{s}]=\mathfrak{s}$ whence, by
\cite[3.2.4]{erdmann}, $\mathfrak{s}\simeq \mathfrak{sl}_2$.

\end{example}
\begin{examples} The examples in \ref{qmat} and \ref{symm} exhibit very different symmetry properties.
In Example~\ref{symm},
there is alternating symmetry in the following sense. For each $\alpha\in S_4$, there are $\C$-automorphisms $\phi_\alpha$ and $\theta_\alpha$ of $A$ such that, for $1\leq i\leq 4$,
$\phi_\alpha(x_i)=x_{\alpha(i)}$ and $\theta_\alpha(x_i)=(-1)^{\sgn \alpha} x_{\alpha(i)}$.
Then $\theta_\alpha$ is a Poisson automorphism. It is enough to check this for the generators $(1~2)$ and $(1~2~3~4)$ of $S_4$, for which
\[\theta_{(1~2)}:x_1\mapsto -x_2,\; x_2\mapsto -x_1,\; x_3\mapsto -x_3,\; x_4\mapsto -x_4\]
and
\[\theta_{(1~2~3~4)}:x_1\mapsto -x_2,\; x_2\mapsto -x_3,\; x_3\mapsto -x_4,\; x_4\mapsto -x_1.\]
Note that $\theta_\alpha(s_2)=s_2$ for all $\alpha\in S_4$, whereas $\theta_\alpha(s_1)=\sgn \alpha~s_1$.
If $\alpha$ is even then $\phi_\alpha=\theta_\alpha$ and if $\alpha$ is odd then $\phi_\alpha$ is a Poisson anti-automorphism.


In Example~\ref{qmat}, there is a well-known  action
of the group
\[H:=\{(h_1,h_2,h_3,h_4)\in(\C^*)^4:h_1h_4=h_2h_3\}\]  on $A=\C[x_1,x_2,x_3, x_4]$, acting as Poisson automorphisms with $x_i\mapsto h_ix_i$ for $1\leq i\leq n$.
If, in \cite[1.2]{Oh4}, we take $A_{2,\Gamma}^{P,Q}=\C[y_1,x_1, y_2,x_2]$ with $q_1=q_2=0, p_1=p_2=-2$ and $\gamma_{12}=-1$ then $A_{2,\Gamma}^{P,Q}$ is Poisson isomorphic to $A$ via the map $y_1\mapsto x_2, x_1\mapsto x_3, y_2\mapsto x_1,x_2\mapsto x_4$. The above action of $H$ on $A$ then corresponds to the action of $H$  on $A_{2,\Gamma}^{P,Q}$ specified in \cite{Oh7} and,  by \cite[2.6]{Oh7}, the $H$-prime Poisson ideals of $A$ are as follows:
$$\begin{array}{c}
0,\\
 x_2A, \quad x_3A, \quad  DA,\\
  x_1A+ x_2A, \quad   x_2A+ x_4A, \quad  x_2A+ x_3A,\quad   x_1A+ x_3A, \quad   x_3A+ x_4A,\\
x_1A+x_2A+x_3A,\;  x_1A+x_2A+x_4A, \;  x_1A+x_3A+x_4A, \;  x_2A+x_3A+x_4A,\\
  x_1A+  x_2A+ x_3A+ x_4A,
\end{array}$$
where $D=x_1x_4-x_2x_3$, the determinant.

This is a special case of a general situation. The multiplicative group $(\C^*)^n$ acts, as algebra automorphisms, on $A$  by the rule
\[(h_1,\dots, h_n).f=f(h_1x_1, \dots, h_n x_n).\]
With $E_{ij}$ as in Notation~\ref{PBnotation}, let $H$ be the subgroup
\[\{(h_1,\dots, h_n):h.E_{ij}=h_ih_jE_{ij}\text{ for } 1\leq i,j\leq n\}.\] If $h=(h_1,\dots, h_n)\in H$
then, for $1\leq i,j\leq n$,
\[\{h.x_i, h.x_j\}=h_ih_j\{x_i,x_j\}=h_ih_j(-1)^{i+j-1}E_{ij}=(-1)^{i+j-1}h.E_{ij}=h.\{x_i,x_j\}.\]
Thus $H$ acts on $A$ by Poisson automorphisms. In Example~\ref{qmat}, $H\simeq (\C^*)^3$, the $3$-torus.
Note that $H$ might be trivial. 
\end{examples}

\begin{notn}
Consider the group \[H^\prime:=\{h\in(\C^*)^n:h.s_i\in \C s_i\text{ and }h.t_i\in \C t_i\text{ for all }i\}.\]
This group is readily calculated from the data and its elements sometimes, but not always, act on $A$ as Poisson automorphisms. In Example~\ref{qmat},
all elements of $H^\prime$ act as Poisson automorphisms. However, in Example \ref{symm},
 where $H^\prime=\{(h_1,h_1,h_1,h_1)\}$, only $(1,1,1,1)$ acts as a Poisson automorphism.
For $1\leq i\leq n-2$, let $\sigma_i:H^\prime\rightarrow \C$ and  $\tau_i:H^\prime\rightarrow \C$ be such that
$h.s_i=\sigma_i(h)s_i$ and $h.t_i=\tau_i(h)t_i$. Let
$\rho:H^\prime\to\C$ be such that, for $h\in H^\prime$, $\rho(h)=\sigma_1(h)\tau_1(h)\dots \sigma_{n-2}(h)\tau_{n-2}(h)$.
 The next result gives a criterion for an element of $H^\prime$ to act as a Poisson automorphism of $A$.
\end{notn}

\begin{prop}\label{whenPauto}
Let $h=(h_1,h_2,\dots,h_n)\in H^\prime$. Then $h$ acts as a Poisson automorphism if and only if
$\rho(h)=h_1\dots h_n$.
\end{prop}
\begin{proof}
For all $g\in A$ and $1\leq i\leq n-2$, $h.\partial_j(g)=h_j^{-1}\partial_j(h.g)$  so,  for $1\leq j\leq n$,
$h.t_i^2\partial_j(s_i/t_i)=h_j^{-1}\sigma_i(h)\tau_i(h)t_i^2\partial_j(s_i/t_i)$. In other words, when $h$ acts on the
$(n-2)\times n$ matrix $E$ whose rows are $t_i^2\nabla(s_i/t_i)$,
the $j$th column gets multiplied by $h_j^{-1}$ and the $i$th row by $\sigma_i(h)\tau_i(h)$.
It follows that, for $1\leq k,\ell\leq n$,
\[h.\{x_k,x_\ell\}=(h_1\dots h_n)^{-1}h_kh_\ell\rho(h)\{x_k,x_\ell\}.\]
As $\{h.x_k,h.x_\ell\}=h_kh_\ell\{x_k,x_\ell\}$, the result follows.
\end{proof}

\begin{examples} Proposition~\ref{whenPauto} is nicely illustrated in Examples~\ref{qmat} and \ref{symm}.
In Example~\ref{qmat}, for $h=(h_1,h_2,h_3,h_1^{-1}h_2h_3)\in H^\prime$, $\sigma_1(h)=h_1h_4=h_2h_3$,
$\sigma_2(h)=h_2$, $\tau_1(h)=1$ and $\tau_2(h)=h_3$  so $\rho(h)=h_1h_2h_3h_4$ for all $h\in H^\prime$.
Here $H^\prime=H$.
In Example~\ref{symm}, for $h=(h_1,h_1,h_1,h_1)\in H^\prime$, $\sigma_1(h)=h_1$,
$\sigma_2(h)=h_1^2$, $\tau_1(h)=1$ and $\tau_2(h)=1$ so,  unless $h_1=1$, $\rho(h)=h_1^3\neq h_1h_2h_3h_4=h_1^4$.
\end{examples}

\begin{lem}\label{constants}
Let $R$ be a commutative noetherian $\C$-algebra that is a domain and let $\delta$ be a $\C$-derivation of $R$. Let $K$ denote the subring of constants, that is
$K=\{r\in R\ :\  \delta(r)=0\}$. Then $K$ is algebraically closed in $R$.
\end{lem}

\begin{proof}
The proof is essentially the same as that of \cite[Lemma 3.1]{JoOh} but with the word \lq algebraic\rq~replacing the word \lq integral\rq~and inserting a leading coefficient $k_n$ that need not be $1$.
\end{proof}

\begin{lem}\label{stP}
Let $P$ be a  proper Poisson prime ideal of $A$. Then $s_i\notin P$ or $t_i\notin P$ for each $i=1,\dots, n-2$.
\end{lem}

\begin{proof}
If $s_i\in P$ and $t_i\in P$ for some $i$ then $P$ is residually null since
$t_i^2\nabla\frac{s_i}{t_i}=t_i\nabla s_i-s_i\nabla t_i\in P$.
\end{proof}

The proof of our main result, Theorem~\ref{AC}, involves the relationship between transcendence degree and heights of prime ideals.
\begin{notn} Let $K$ be a field, $A$ be an integral domain which is also an affine $K$-algebra, $Q(A)$ be the field of quotients of $A$ and $L$ be a field extension of $K$.
We shall denote the Krull dimension of an affine $K$-algebra $A$ by $\dim(A)$ and the transcendence degree of $L$ by $\trdeg_K(L)$. Following \cite[Chapter 6]{commview} we extend the latter notation to $A$ by taking $\trdeg_K(A)$ to be the number of elements in any maximal algebraically independent set of elements in $A$. By \cite[Corollary 14.29]{Sha}
and \cite[Theorem 6.35]{commview}, $\trdeg_K(Q(A))=\dim(A)=\trdeg_K(A)$. Note also that any algebraically independent set of elements in $A$ can be extended to a maximally algebraically independent set in $A$, see \cite[Example 6.4 and Remark 6.6]{commview}. We shall simply write $\trdeg(A)$ for $\trdeg_K(A)$ if no confusion arises.
By \cite[Corollary 14.32]{Sha},
\[\hgt(P)+\dim(A/P)=\dim(A).\]
\label{trdimht}
\end{notn}

\begin{notn}
Let $\Gamma$ be the set of all sequences $\gamma=((\gamma_1,\delta_1),\dots,(\gamma_{n-2},\delta_{n-2}))$, of length $n-2$, in $\{0,1\}\times\{0,1\}$. Call an element $\gamma$ of $\Gamma$ {\it dense} if, for each $i=1,\dots, n-2$, $(\gamma_i,\delta_i)\neq (0,0)$. To each $\gamma\in \Gamma$, we associate a finite subset $S_\gamma$ of
$\{s_1,\dots,s_{n-2},t_1,\dots,t_{n-2}\}$, a finite subset
$V_\gamma=\{v_1,\dots,v_{n-2}\}$ of $B$, a multiplicatively closed subset $M_\gamma$ of $A$ and a localization $A_\gamma$ of $A$ as follows:
\[
s_i\in S_\gamma\Leftrightarrow\gamma_i=1\text{ and }t_i\in S_\gamma\Leftrightarrow\delta_i=1,\]
\[v_i= \begin{cases} s_i/t_i \text{ if }\delta_i=1\\ t_i/s_i \text{ otherwise},\end{cases}\]
 $M_\gamma$ is the multiplicative closed subset of $A$ generated by the elements of $S_\gamma$, and $A_\gamma$ is the localization $M_\gamma^{-1}A$. Note that if $\gamma$ is dense then $v_i\in A_\gamma$ for each $i=1,\dots,n-2$. In this case, denote by $C_\gamma$ the subalgebra of $A_\gamma$ generated by $v_1,\dots, v_{n-2}$.
Since $s_1/t_1,\dots, s_{n-2}/t_{n-2}$ are algebraically independent over $\C$, the transcendence degree of $C_\gamma$ is $n-2$.

For example, in Example~\ref{qmat}, if $\gamma=\{(0,1),(1,0)\}$ then $S_\gamma=\{t_1,s_2\}=\{1,x_2\}$ and $v_\gamma=\{s_1/t_1,t_2/s_2\}=
\{x_1x_4-x_2x_3,x_3/x_2\}$.
\end{notn}

\begin{notn}
For $P\in\Pspec(A)$, let $\gamma(P)=((\gamma_1,\delta_1),\dots,(\gamma_{n-2},\delta_{n-2}))$ be the sequence such that, for each $i$,
$\gamma_i=0\Leftrightarrow s_i\in P$ and $\delta_i=0\Leftrightarrow t_i\in P$. For example, in Example~\ref{qmat},
if $P=x_1A+x_3A$ then $\gamma(P)=\{(0,1),(1,0)\}$, $S_{\gamma(P)}=\{s_2,t_1\}=\{x_2,1\}$ and $V_{\gamma(P)}=\{s_1/t_1,t_2/s_2\}=
\{x_1x_4-x_2x_3,x_3/x_2\}$.
\end{notn}

The next lemma amounts to observing some restrictions on $\gamma(P)$.
\begin{lem}\label{restrictions}
Let $P$ be a Poisson prime ideal of $A$.
\begin{enumerate}
\item If $P$ is proper then $\gamma(P)$ is dense.
\item If $t_i=1$ for some $i$ then, in $\gamma(P)$, $\delta_i=1$ and, in $V_{\gamma(P)}$, $v_i=s_i$.
\end{enumerate}
\end{lem}
\begin{proof}
(1) holds because, by Lemma~\ref{stP}, we cannot have $s_i\in P$ and $t_i\in P$ for any $i$ and (2) holds because
$t_i\notin P$.
\end{proof}

\begin{rem}\label{sometone}
The converse to (1) is false as can be seen from Example~\ref{qmat} where, for the residually null Poisson prime ideal $P=x_1A+x_2A+x_4A$, $\gamma(P)=\{(0,1),(0,1)\}$ is dense.
\end{rem}

\begin{notn}\label{strategy}
 The Poisson spectrum  $\Pspec A$ can be partitioned using $\Gamma$. For $\gamma\in \Gamma$, let \[\Pspec_\gamma A=\{P\in \Pspec A |S_\gamma=S_{\gamma(P)}\}.\]
The set $\Pspec_\gamma A$ may be empty. Indeed, by Lemma~\ref{restrictions}(2), if $t_i=1$ for some
$i$ then  $\Pspec_\gamma A=\varnothing$ whenever $\delta_i=0$.

Our strategy in attempting to understand $\Pspec A$ is based on the following:
\begin{enumerate}
\item $\Pspec A$ is the disjoint union of the subsets $\Pspec_\gamma A$ taken over $\gamma\in \Gamma$.
\item By standard localization theory, if $P\in \Pspec_\gamma A$ then $PA_{\gamma}$ is a Poisson prime ideal of $A_{\gamma}$ and $P=A\cap PA_{\gamma}$.
\item When $\gamma$ is dense, Theorem~\ref{AC} below determines the Poisson prime ideals of $A_{\gamma}$ in terms of prime ideals of the polynomial algebra $C_{\gamma}$.
\item When $\gamma$ is not dense, every Poisson prime ideal in $\Pspec_\gamma A$ is residually null.
\end{enumerate}
 \end{notn}

The next result determines the proper Poisson prime ideals in $A_\gamma$ when $\gamma$ is dense.

\begin{thm}\label{AC}
Let $\gamma\in\Gamma$ be dense.
\begin{enumerate}
\item Let $I$ be an ideal of $C_\gamma$ and let  $Q$ be a  prime ideal of $A_\gamma$ that is minimal over $IA_\gamma$. Then $Q$ is a Poisson prime ideal
of $A_\gamma$.
\item
If $Q$ is a nonzero proper Poisson prime ideal of $A_\gamma$ then $\hgt(Q)=\hgt(Q\cap C_\gamma)$.
\item
If $Q$ is a nonzero proper Poisson prime ideal of $A_\gamma$ then $Q$ is a minimal prime ideal over $(Q\cap C_\gamma)A_\gamma$.
\end{enumerate}
\end{thm}

\begin{proof}
(1) Since
$v_i= s_i/t_i$ or $v_i=t_i/s_i$ and $\nabla\frac{t_i}{s_i}=-s_i^{-2}(t_i^2\nabla\frac{s_i}{t_i})$, the subalgebra $C_\gamma$ is contained in the Poisson centre of $A_\gamma$. Hence $IA_\gamma$ is a Poisson ideal of $A_\gamma$ and,  by \cite[1.4]{Oh10}, every prime ideal of $A_\gamma$ minimal over $IA_\gamma$ is Poisson.

(2)
By Noether's Normalization Theorem, as stated in \cite[14.14]{Sha}, there exist non-negative integers $m,d$, with
$d\leq m$, and $y_1,\dots, y_m\in C_\gamma$ such that
$y_1,\dots, y_m$ are algebraically independent over $\C$, $C_\gamma$ is integral over $\C[y_1,\dots, y_m]$ and
$(Q\cap C_\gamma)\cap \C[y_1,\dots, y_m]=\sum_{i=d+1}^{m}\C[y_1,\dots, y_m]y_i$. Note that $m=\trdeg_\C(C_\gamma)=n-2$ and $d=\trdeg_\C(C_\gamma/Q\cap C_\gamma)$. The algebraically independent subset $\{y_1,\dots, y_{n-2}\}$ of $A_\gamma$ can be extended to a maximal algebraically independent subset $\{y_1,\dots, y_{n-2},z_1, z_2\}$ of $A_\gamma$. Thus $A_\gamma$
is algebraic over $\C[y_1,\dots, y_{n-2}, z_1, z_2]$.
As $y_{d+1}, \dots, y_{n-2}\in Q$, $A_\gamma/Q$ is algebraic over $\C[y_1+Q,\dots, y_{d}+Q, z_1+Q, z_2+Q]$.  It follows that
\[\trdeg(A_\gamma/Q)\leq d+2=\trdeg(C_\gamma/Q\cap C_\gamma)+2,\]
and, from Notation \ref{trdimht}, that
\begin{equation}\label{height}
\hgt(Q)\geq \hgt(Q\cap C_\gamma).
\end{equation}
Now suppose that $\hgt(Q)> \hgt(Q\cap C_\gamma)$. Then
$\trdeg(A_\gamma/Q)\leq \trdeg(C_\gamma/Q\cap C_\gamma)+1$. Let $T$ be the set of all nonzero elements of the integral domain $C_\gamma/Q\cap C_\gamma$
and let $K=T^{-1}(C_\gamma/Q\cap C_\gamma)$ be the quotient field of $C_\gamma/Q\cap C_\gamma$.
 Thus $T^{-1}(A_\gamma/Q)$ is an affine $K$-algebra.
Let $L$ be the quotient field of $A_\gamma/Q$ which is also the quotient field of $T^{-1}(A_\gamma/Q)$.
Then $\trdeg_\C(L)\leq \trdeg_\C(K)+1$ and, by \cite[12.56]{Sha}, $\trdeg_K L\leq 1$. Hence there exists $w\in T^{-1}(A_\gamma/Q)$ such that
$T^{-1}(A_\gamma/Q)$ is algebraic over $K[w]$. Moreover $C_\gamma$ is contained in the Poisson centre of $A_\gamma$, whence the Poisson bracket in $K[w]$ is trivial.
The constant subring of the hamiltonian $\ham w$ contains  $K[w]$ and,  by Lemma~\ref{constants}, it contains $T^{-1}(A_\gamma/Q)$. Hence, for any $b\in T^{-1}(A_\gamma/Q)$, the constant subring of
$\ham b$ contains $K[w]$ and, again by Lemma~\ref{constants}, $\ham b=0$. Thus
the  Poisson bracket in $T^{-1}(A_\gamma/Q)$ is trivial, which is
impossible since $Q$ is proper. Therefore we must have equality in \eqref{height}, that is,
$\hgt(Q)=\hgt(Q\cap C_\gamma).$

(3) Let $Q^\prime$ be a minimal prime ideal of
$(Q\cap C_\gamma) A_\gamma$ such that $Q^\prime\subseteq Q$. Then $Q^\prime$ is Poisson prime, by (1), and is proper. Suppose that $Q^\prime\neq Q$. Then $\hgt(Q^\prime)<\hgt(Q)$ and \[\hgt(Q^\prime)=\hgt(Q^\prime\cap C_\gamma)=\hgt(Q\cap C_\gamma)=\hgt(Q),\] a clear contradiction.
\end{proof}

\begin{rem}
Let $P$ be a proper Poisson prime ideal of $A$ and let $\gamma=\gamma(P)$. By Theorem~\ref{AC}, $PA_\gamma$ is a minimal prime ideal over $PA_\gamma\cap C_\gamma$ and
$\hgt(P)=\hgt(PA_\gamma)=\hgt (PA_\gamma\cap C_\gamma)$. Denote by $\Phgt(P)$ the maximal length $\ell$ of a chain of distinct Poisson prime ideals
\[0=P_0\subset P_1\subset P_2\subset\dots\subset P_\ell=P\]
of $A$. Clearly  $\Phgt(P)\leq \hgt(P)$. But, if $j=\hgt P=\hgt PA_\gamma=\hgt(PA_\gamma\cap C_\gamma)$ then a chain of prime ideals of $C_\gamma$ of length $j$ inside $PA_\gamma\cap C_\gamma$ gives rise to chain of Poisson prime ideals of $A$ of length $j$ inside $P$ so
$\Phgt(P)\geq \hgt(PA_\gamma\cap C_\gamma)=\hgt(P)$, whence $\Phgt(P)=\hgt(P)$. There are many examples of residually null Poisson prime ideals $P$ for which $\Phgt(P)<\hgt(P)$, for example the symmetric algebra of $\mathfrak{sl}_2$ with the Kirillov-Kostant-Souriau bracket, where the unique Poisson maximal ideal $M$ has $\Phgt(M)=2$ and $\hgt(M)=3$, see \cite[Example 4.1]{JoOh}. For an example where $\Phgt(P)<\hgt(P)$ and $P$ is not residually prime, see \cite[Remark 5.13]{PoissonOre}.
\end{rem}

\begin{example}\label{qmat2}
To illustrate Theorem~\ref{AC} and the strategy outlined in Notation \ref{strategy}, we return to Example~\ref{qmat} and describe all Poisson prime ideals. It is easy to see that
the residually null Poisson prime ideals in this case are the height two prime ideal $x_2A+x_3A$, the height three prime ideals $x_1A+x_2A+x_4A$ and $x_1A+x_3A+x_4A$ and all prime ideals containing one, or more, of these.

Let $D=x_1x_4-x_2x_3=s_1$, the determinant.
The dense subsets of $\Gamma$ for which $\Pspec_\gamma A$ can be non-empty and the corresponding sets $S_\gamma$ and $V_\gamma$ are:
\begin{align*}
\gamma_1&:=\{(0,1),(0,1)\},& S_{\gamma_1}&=\{1,x_3\},& V_{\gamma_1}&=\{D,x_2/x_3\},\\
\gamma_2&:=\{(0,1),(1,1)\},& S_{\gamma_2}&=\{x_2,1,x_3\},& V_{\gamma_2}&=\{D,x_2/x_3\},\\
\gamma_3&:=\{(1,1),(0,1)\},& S_{\gamma_3}&=\{D,1,x_3\},& V_{\gamma_3}&=\{D,x_2/x_3\},\\
\gamma_4&:=\{(1,1),(1,1)\},& S_{\gamma_4}&=\{D,x_2,1,x_3\},& V_{\gamma_4}&=\{D,x_2/x_3\},\\
\gamma_5&:=\{(1,1),(1,0)\},& S_{\gamma_5}&=\{D,x_2,1\},& V_{\gamma_5}&=\{D,x_3/x_2\},\\
\gamma_6&:=\{(0,1),(1,0)\},& S_{\gamma_6}&=\{x_2,1\},& V_{\gamma_6}&=\{D,x_3/x_2\}.
\end{align*}

Consequently $C_\gamma=C_1:=\C[D,x_2/x_3]$ or $C_\gamma=C_2:=\C[D,x_3/x_2]$.

Let $P$ be a proper Poisson prime ideal of $A$ and let $\gamma=\gamma(P)$. Then $PA_{\gamma}$ is minimal over
$PA_{\gamma}\cap C_{\gamma}$. Suppose that  $C_{\gamma}=C_1$ so that
$\gamma=\gamma_i$ for some $i$ with $1\leq i\leq 4$.  The prime ideals of $C_1$ are $0$, the principal ideals $fC_1$, where $f$ is irreducible in $C_1$ and
the maximal ideals $(D-\lambda)C_1+(x-\mu)C_1$, where $x=x_2/x_3$.

If $i=1$ then $D,x_2\in P$ and thus  $PA_{\gamma}\cap C_1=xC_1+DC_1$  so  $P$ is minimal over $x_2A+DA=x_2A+x_1x_4A$. It follows that  $P=x_2A+x_1A$ or $P=x_2A+x_4A$. In both cases $A/P\simeq \C[y,z]$ with $\{y,z\}=yz$.

If $i=2$ then $D\in P$, and
$PA_{\gamma}\cap C_1=DC_1$ or $PA_{\gamma}\cap C_1=DC_1+(x-\lambda)C_1$ for some non-zero $\lambda\in \C$. In this case either $P=DA$ or $P=DA+(x_2-\lambda x_3)A$. In the latter case, $A/P$ is isomorphic to
$\C[x_1,x_3,x_4]/(x_1x_4-\lambda x_3^2)$ with the bracket induced by the Poisson bracket on $\C[x_1,x_3,x_4]$ such that $\{f,g\}=x_3 \Jac(f,g,x_1x_4-\lambda x_3^2)$ for all $f,g\in \C[x_1,x_3,x_4]$. It is easy to see, using \cite[Theorem 3.8]{JoOh}, that the non-zero Poisson prime ideals of $A/P$ are residually null.

If $i=3$ then $PA_{\gamma}\cap C_1=xC_1$ or $PA_{\gamma}\cap C_1=xC_1+(D-\lambda)C_1$ for some non-zero $\lambda\in \C$ so either $P=x_2A$ or $P=x_2A+(D-\lambda)A$ and, in the latter case, $A/P$ is isomorphic to
$\C[x_1^{\pm 1},x_3]$ with $\{x_1,x_3\}=x_1x_3$.

If $i=4$ then $PA_{\gamma}\cap C_1=0$ or $PA_{\gamma}\cap C_1=fC_1$, for some irreducible $f\in C_1$ that is not
an associate of $D$ or $x$,  or $PA_{\gamma}\cap C_1=(x-\mu)C_1+(D-\lambda)C_1$ for some non-zero $\mu,\lambda\in \C$. In the third of these cases, $P=(x_2-\mu x_3)A+(D-\lambda)A$ and
$A/P\simeq \C[x_1,x_3,x_4]/(x_1x_4-\mu x_3^2-\lambda)$. In the second case, $f$ remains irreducible
in the polynomial extensions $\C[x,D,x_3]=\C[x,x_1x_4,x_3]$ and $\C[x,x_1x_4,x_3,x_1]$ and in the localization $T$ of the latter at the multiplicatively closed subset generated by $x_1$ and $S_\gamma$. It follows that $f$ is irreducible in $A_{\gamma}$ since $A_{\gamma}$ is a subalgebra of $T$. Hence  if $j$ is the minimal non-negative integer such that $fx_3^j\in A$ and $g=fx_3^j$
then $g$ is irreducible in $A$ and $P=fA_{\gamma}\cap A=gA$. Examples of Poisson prime ideals arising in this way include the principal ideals generated by $g_0=D-\lambda$, $g_1=(x_2-\lambda x_3)=x_3(x-\lambda)$, where $\lambda\in \C\backslash\{0\}$, $g_2=x_1x_3x_4-x_2x_3^2+x_2=
Dx_3+x_2=(D+x)x_3$, $g_3=x_1x_2x_4-x_2^2x_3+x_3=Dx_2+x_3=
(xD+1)x_3$, $g_4=(D^2-x^3)x_3^3=D^2x_3^3-x_2^3$ and $g_5=D^2x_2^3-x_3^3=x_3^3(D^2x^3-1)$. Here the pairs  $g_2, g_3$ and $g_4,g_5$ show how the choice of $v_2$, which is not symmetric between $x_2$ and $x_3$, takes account of the inherent symmetry between $x_2$ and $x_3$. In general, if $f(D,x^{-1})$ is irreducible in $C_2$, where $x^{-1}=x_3x_2^{-1}$, then there is
an irreducible polynomial $g(D,x)$ such that $g(D,x)=x^k f(D,x^{-1})$ for some $k\geq 0$.

The symmetry between $x_2$ and $x_3$ is more explicit in the analysis for $\gamma_5$ and $\gamma_6$,
which are analogous to $\gamma_3$ and $\gamma_1$ respectively. Here the Poisson prime ideals  are $P=x_3A$ or $P=x_3A+(D-\lambda)A$, $\lambda\in \C\backslash\{0\}$, for $\gamma_5$, and $P=x_3A+x_1A$ or $P=x_3A+x_4A$ for $\gamma_6$.
\end{example}

In the case where $t_1=t_2=\dots=t_{n-2}=1$, if $\gamma$ is such that $\Pspec A_\gamma$ is non-empty, then, by Lemma~\ref{restrictions},
each $\delta_i=1$, each $v_i=s_i$ and  $C_\gamma=\C[s_1,s_2,\dots,s_{n-2}]$ which, under our working assumption,
is a polynomial subalgebra of $A$.

\begin{cor}\label{tallone}
Suppose that $t_1=t_2=\dots=t_{n-2}=1$, let $C=\C[s_1,s_2,\dots,s_{n-2}]$ and let $P$ be a proper Poisson prime ideal of $A$. Then there exists a prime ideal $Q$ of $C$ such that $P$ is a minimal prime ideal over $QA$.
\end{cor}
\begin{proof}
Let $\gamma=\gamma(P)$. In this case, $C\subseteq A \subseteq A_\gamma$. By Theorem~\ref{AC}, $PA_\gamma$ is a minimal prime of $A_\gamma$ over $PA_\gamma\cap C$ and it follows easily that $P$ is a minimal prime of $A$ over $(P\cap C)A$.
\end{proof}

\begin{example}\label{x2x3}
In Example~\ref{qmat2}, each irreducible factor $f$ of $C_1$ leads to a single principal Poisson prime ideal. This is not always the case. 
For example, consider the case where $n=4$, $s_1=x_1x_4-x_2x_3$, $s_2=x_2x_3$ and $t_1=t_2=1$. Thus Corollary~\ref{tallone} applies. Note that $s_2$ and $s_1+s_2$ are irreducible in $C$ but not in $A$. This gives rise to four height one Poisson prime ideals of $A$, $x_iA$ for $1\leq i\leq 4$. 
\end{example}

\begin{example}\label{s} This example illustrates the situation where
$A$ has an element of the form $\lambda_i s_i-\mu_i t_i$ for two different values of $i$.
This gives rise to residually null Poisson prime ideals.
Let $n=4$ and let $s_1=x_1+x_2+x_3+x_4$, $t_1=1$, $s_2=x_1+x_4$ and $t_2=x_2+x_3$.
Then $s_1=\lambda_1 s_1-\mu_1 t_1=\lambda_2 s_2-\mu_2 t_2$, where $\lambda_1=\lambda_2=1$,
 $\mu_1=0$ and $\mu_2=-1$.
 The Poisson bracket on $A$ in this example is given by
\[\begin{aligned}
\{x_1,x_2\}&=s_1, &\{x_1,x_3\}&=-s_1,&\{x_1,x_4\}&=0, \\
\{x_2,x_3\}&=0, &\{x_2,x_4\}&=s_1,&\{x_3,x_4\}&=-s_1.
\end{aligned}
\]
Here the height one prime ideal $P=s_1A$ is residually null Poisson and $\gamma(P)=((0,1),(1,1))$ is dense. Notice that $PA_\gamma\cap C_\gamma$ contains both $v_1=s_1$
and $v_2+1=(x_2+x_3)^{-1}s_1$ and that $\hgt(PA_\gamma\cap C_\gamma)=2$ whereas $\hgt(PA_\gamma)=1$.
\end{example}

\section{Poisson primitive spectra}

\begin{notn}\label{Ppnot}
For each $i=1,\dots, n-2$ and for each $(\lambda_i,\mu_i)\in\C^2\backslash\{(0,0)\}$, set \[f^i_{\lambda_i,\mu_i}=\lambda_i s_i-\mu_i t_i.\]  Observe that
\begin{equation}\label{t2nabla}
t_i^2\nabla\frac{s_i}{t_i}=\begin{cases}\lambda_i^{-1}(t_i\nabla f^i_{\lambda_i,\mu_i}-f^i_{\lambda_i,\mu_i}\nabla t_i) \text{ if }\lambda_i\neq0,\\
\mu_i^{-1}(s_i\nabla f^i_{\lambda_i,\mu_i}-f^i_{\lambda_i,\mu_i}\nabla s_i) \text{ if }\mu_i\neq0.\end{cases}
\end{equation}
\end{notn}

\begin{lem}\label{MpPoisson}
For $p=(\alpha_1,\alpha_2, \dots, \alpha_n)\in\C^n$, let $M_p=(x_1-\alpha_1)A+(x_2-\alpha_2)A+\dots+(x_n-\alpha_n)A$
and let $g_i=f_{t_i(p), s_i(p)}^i$, $1\leq i\leq n-2$. Then $M_p$ is
a Poisson ideal  if and only if one of the following conditions holds:
\begin{enumerate}  \item $p$ is a common zero of $s_i$ and $t_i$ for some $i$;
\item   $g_1,\dots,g_{n-2}$ are algebraically dependent over $\C$;
\item
$p$ is a singular point of the affine variety determined by $g_1,\dots,g_{n-2}$.
\end{enumerate}
\end{lem}

\begin{proof}
It suffices to show that $M_p$ is
a Poisson ideal if (1) holds or if (1) fails and (2) holds  and that if (1) and (2) fail then $M_p$ is
a Poisson ideal if and only if (3) holds.

Suppose that (1) holds. Thus $s_i(p)=t_i(p)=0$ for some $i=1,\dots, n-2$ and
\[\left(t_i^2\nabla\frac{s_i}{t_i}\right)(p)=(t_i\nabla s_i-s_i\nabla t_i)(p)=0.\]
Let $1\leq k,\ell\leq n$. Then
\[\{x_k,x_\ell\}(p)=\left|\begin{matrix}
e_k\\e_\ell\\ \left(t_1^2\nabla\frac{s_1}{t_1}\right)(p)\\ \vdots \\ \left(t_{n-2}^2\nabla\frac{s_{n-2}}{t_{n-2}}\right)(p)
\end{matrix}\right|=0\]
so $\{x_k,x_\ell\}\in M_p$ which is therefore a Poisson ideal.

We can now assume that (1) fails.
Let $i=1,\dots, n-2$ and set
$\lambda_i=t_i(p)$ and $\mu_i=s_i(p)$.
Thus
$(\lambda_i,\mu_i)=(t_i(p),s_i(p))\neq (0,0)$
and $g_i(p)=0$. By \eqref{t2nabla},
\[\left(t_i^2\nabla\frac{s_i}{t_i}\right)(p)=\left(\nabla g_i\right)(p).\]

Now suppose that (1) fails but (2) holds.
Thus $(\lambda_i,\mu_i)\neq(0,0)$ for $1\leq i\leq n-2$ and $g_1,\dots,g_{n-2}$
are algebraically dependent over $\C$. For  $1\leq k,\ell\leq n-2$,
\[\{x_k,x_\ell\}(p)=\left|\begin{matrix}
e_k\\e_\ell\\ \left(t_1^2\nabla\frac{s_1}{t_1}\right)(p)\\ \vdots \\ \left(t_{n-2}^2\nabla\frac{s_{n-2}}{t_{n-2}}\right)(p)
\end{matrix}\right|=
\left|\begin{matrix}
e_k\\e_\ell\\ \left(\nabla g_1\right)(p)\\ \vdots \\ \left(\nabla g_{n-2}\right)(p)
\end{matrix}\right|=0\]
 by Theorem~\ref{algdep}, applied with the algebraic dependent elements $g_1,\dots,g_{n-2}$ in place of $f_1,\dots,f_{n-2}$.  Thus $M_p$ is a Poisson ideal.

Finally, suppose that (1) and (2) fail. As (1) fails, $(\lambda_i,\mu_i)\neq (0,0)$ for $1\leq i\leq n-2$. As (2) fails, $g_{1},\dots, g_{n-2}$ are algebraically independent over $\C$ and so
 the dimension of the  affine variety
$\mathcal{V}(g_{1},\dots, g_{n-2})$ that they determine
is two. Then $M_p$ is a Poisson ideal of $A$ if and only if, for all $k,\ell$, 
$$0=\{x_k,x_\ell\}(p)=\left|\begin{matrix}
e_k\\e_\ell\\ \left(t_1^2\nabla\frac{s_1}{t_1}\right)(p)\\ \vdots \\ \left(t_{n-2}^2\nabla\frac{s_{n-2}}{t_{n-2}}\right)(p)
\end{matrix}\right|=
\left|\begin{matrix}
e_k\\e_\ell\\ \left(\nabla g_1\right)(p)\\ \vdots \\ \left(\nabla g_{n-2}\right)(p)
\end{matrix}\right|$$
if and only if  the  determinants of all $(n-2)\times(n-2)$-submatrices of
$\begin{pmatrix}
\nabla g_1\\ \vdots \\ \nabla g_{n-2}
\end{pmatrix}$ vanish at  $p$ if and only if
 $p$ is a singular point of
$\mathcal{V}(g_1,\dots,g_{n-2})$.
Note that here we are using the Jacobian Criterion in a more general form, for example \cite[Corollary 16.20]{eisenbud}, than a form which applies to generators of a prime or reduced ideal.

This completes the proof.
\end{proof}

\begin{lem}\label{fA2}
Suppose that the parameters $\lambda_i$ and $\mu_i$ are such that  the ideal $I:=f^{1}_{\lambda_1,\mu_1}A+\dots+f^{n-2}_{\lambda_{n-2},\mu_{n-2}}A$ is a proper ideal of $A$ and let $P$
be a minimal prime ideal of $I$.
\begin{enumerate}
\item $P$ is a Poisson prime ideal with height less than or equal to $n-2$.
\item If $P$ is residually null then it is not Poisson primitive.
\item If $P$ is proper then $\text{ht}P=n-2$ and $P$ is locally closed and Poisson primitive.
\end{enumerate}
\end{lem}

\begin{proof}
(1) It follows from  \eqref{t2nabla} that $I$ is a Poisson ideal so $P$ is a Poisson prime ideal. The height of $P$ is less than or equal to $n-2$ by
\cite[15.4]{Sha}.

(2) By Lemma~\ref{poissonprimitive}(1), any residually null Poisson primitive ideal is maximal and hence  has height $n$. By (1), $P$ is not Poisson primitive.

(3) Let $\gamma=\gamma(P)$ which is dense by Lemma~\ref{restrictions}(1). Let $1\leq i\leq n-2$.
Suppose that $t_i\notin P$.  Then $t_i\in S_\gamma$,
$v_i=s_i/t_i$, and $\lambda_iv_i-\mu_i=t_i^{-1}f^{i}_{\lambda_i,\mu_i}\in PA_\gamma$. Note that
$\lambda_i\neq 0$, otherwise $0\neq \mu_i t_i =-f^{i}_{\lambda_i,\mu_i}\in P$.
Similarly if $t_i\in P$ then $s_i\in S_\gamma$, $v_i=t_i/s_i$, $\lambda_i-\mu_iv_i=s_i^{-1}f^{i}_{\lambda_i,\mu_i}\in PA_\gamma$  and $\mu_i\neq 0$.
Therefore $PA_\gamma\cap C_\gamma$ must be the maximal ideal of $C_\gamma$ generated by the elements
$m_i$, where $m_i=v_i-\frac{\mu_i}{\lambda_i}$ if $t_i\notin P$ and $m_i=v_i-\frac{\lambda_i}{\mu_i}$ if $t_i\in P$.  By Theorem~\ref{AC}(2), $n-2=\hgt PA_\gamma=\hgt P$, and hence $P$ is locally closed and Poisson primitive by Lemma~\ref{poissonprimitive}(3).
\end{proof}

We next determine the Poisson primitive ideals of $A$ and establish the Poisson Dixmier-Moeglin equivalence.

\begin{thm}\label{primitiveclass}
The Poisson primitive ideals of $A$ are
the Poisson maximal ideals, as specified in Lemma 3.2, and the proper Poisson ideals that are minimal prime ideals  of a proper ideal
$f^{1}_{\lambda_1,\mu_1}A+\dots+f^{n-2}_{\lambda_{n-2},\mu_{n-2}}A$, as specified in Lemma~\ref{fA2}.
Moreover $A$ satisfies the Poisson Dixmier-Moeglin equivalence.
\end{thm}

\begin{proof}
 Poisson maximal ideals are always Poisson primitive, so it follows from Lemma~\ref{fA2}(3) that the listed ideals are Poisson primitive.
Let $P$ be any Poisson primitive ideal and let $1\leq i\leq n-2$.
Since $s_i/t_i$ is a Poisson central element of the quotient field $Q(A)$ of $A$, it follows, from \cite[1.10]{Oh4}, that there exists $(\lambda_i, \mu_i)\in\C^2\backslash\{(0,0)\}$ such that $P$ contains
$f_{\lambda_i,\mu_i}^i=\lambda_i s_i-\mu_i t_i$.
 Thus $P$ contains the ideal $f^{1}_{\lambda_1,\mu_1}A+\dots+f^{n-2}_{\lambda_{n-2},\mu_{n-2}}A$.
 If $\text{ht}P\geq n-1$ then $P$ is residually null by \cite[Proposition 3.2]{JoOh}. If $P$ is residually null then, by Lemma~\ref{poissonprimitive}(1), $P$ is a Poisson maximal ideal. Hence we may assume that $\hgt P\leq n-2$ and that $P$ is proper. Let $Q$ be a minimal prime ideal of $f^{1}_{\lambda_1,\mu_1}A+\dots+f^{n-2}_{\lambda_{n-2},\mu_{n-2}}A$
 such that $Q\subseteq P$. By \cite[1.4]{Oh10}, $Q$ is Poisson. If $Q$ is residually null then so is $P$,  a contradiction. Hence $Q$ is proper and $\text{ht}Q=n-2$ by Lemma~\ref{fA2}(2). It follows that $P=Q$ is a
 minimal prime ideal of the ideal  $f^{1}_{\lambda_1,\mu_1}A+\dots+f^{n-2}_{\lambda_{n-2},\mu_{n-2}}A$, as specified in Lemma~\ref{fA2}.

To establish the Poisson Dixmier-Moeglin equivalence, let $P$ be a rational Poisson prime ideal. Let $1\leq i\leq n-2$.
As $s_i/t_i\in \Pz(Q(A))$, $P$ contains
$f^i_{\lambda_i,\mu_i}$ for some $(\lambda_i,\mu_i)\in\C^2\backslash\{(0,0)\}$, and therefore $P$ contains a proper ideal of the form $f^{1}_{\lambda_1,\mu_1}A+\dots+f^{n-2}_{\lambda_{n-2},\mu_{n-2}}A$. If $P$ is residually null then
$\C=\Pz(Q(A/P))=Q(A/P)$ so $P$ is a Poisson
maximal ideal and hence is locally closed. If  $P$ is proper then $\hgt P\leq n-2$ by \cite[Proposition 3.2]{JoOh} and, by Lemma~\ref{fA2}(3), $P$ is locally closed.
\end{proof}

\begin{cor}
Suppose that $t_i=1$ for each $i$ and that $s_1,\dots,s_{n-2}$ are  algebraically independent.
Let $(\mu_1,\dots,\mu_{n-2})\in \C^{n-2}$ be such that $P:=(s_1-\mu_1)A+\dots+(s_{n-2}-\mu_{n-2})A$ is a prime ideal of $A$. Let  $X\subset\C^n$ be the variety determined by $P$. Then $P$ is Poisson prime. Moreover  $X$ is nonsingular if and only if $A/P$ is Poisson simple.
 \end{cor}

\begin{proof}
In Notation~\ref{Ppnot}, let $\lambda_i=1$ so that $f_{\lambda_i,\mu_i}^i=s_i-\mu_i$. By Lemma~\ref{fA2}(1), $P$ is Poisson. By Lemma~\ref{restrictions}(2), $C_{\gamma(P)}=\C[s_1,\dots,s_{n-2}]$ so it follows, by Theorem~\ref{AC}(2), that
  $\hgt P=n-2$. Hence $\dim X=2$. Let $Q$ be a Poisson primitive ideal of $A$ such that $P\subseteq Q$. By Theorem~\ref{primitiveclass} and Lemma~\ref{fA2}, either $\hgt Q=n-2$, in which case $Q=P$, or
$Q$ is the maximal ideal corresponding to a singularity of $X$. Hence $A/P$ has no nonzero Poisson primitive ideal if
and only if $X$ is nonsingular. The result now follows from Lemma~\ref{poissonprimitive}(2).
\end{proof}

\bibliographystyle{amsplain}

\begin{thebibliography}{1}
\bibitem{BGl}
K. A. Brown and K. R. Goodearl, Lectures on Algebraic quantum
groups, Birkh\"{a}user (Advanced Courses in Mathematics CRM
Barcelona), Basel-Boston-Berlin, 2002.

\bibitem{eisenbud} D. Eisenbud, Commutative Algebra with a View Toward Algebraic Geometry,
Springer Verlag, New York 1995.


\bibitem{erdmann} K. Erdmann and M. J. Wildon, Introduction to Lie
Algebras, Springer, London, 2006.


\bibitem{GKZ} I. M. Gelfand, M. Kapranov and A. Zelevinsky,
Discriminants, Resultants, and Multidimensional Determinants, Birkh\"{a}user Boston 1994.



\bibitem{goodsemiclass} K. R. Goodearl, {Semiclassical limits of
quantized coordinate rings}, in \emph{Advances in Ring Theory} (D.V. Huynh and S. Lopez-Permouth, Eds.), Basel (2009) Birkh\"{a}user, 165--204 and at
arXiv:math.QA/0812.1612v1.

\bibitem{Good3}
K.~R. Goodearl, \emph{A Dixmier-Moeglin equivalence for Poisson algebras with
  torus actions}, in Algebra and Its Applications (D. V. Huynh, S. K. Jain, and
  S. R. Lopez-Permouth, Eds.) Contemporary Mathematics \textbf{419} (2006), 131--154.




\bibitem{hodgepedoe} W. V. D. Hodge and D. Pedoe, Methods of Algebraic Geometry Volume 1, Cambridge Mathematical Library (reissue), Cambridge University Press, Cambridge 1994.

\bibitem{GMP} J. Grabowski, G. Marmo and A. M. Perelomov, \emph{Poisson structures: towards a classification}
Modern Phys. Lett. A8 \textbf{18} (1993), 1719-1733.

\bibitem{JoOh}
D. A. Jordan and Sei-Qwon Oh, \emph{Poisson brackets and Poisson spectra in
  polynomial algebras}, Contemporary Mathematics \textbf{562} (2012), 169--187.

\bibitem{PoissonOre} D. A. Jordan, \emph{Ore Extensions and Poisson algebras}, arXiv:math.RA/1212.4063 (2012).


\bibitem{Oh4}
Sei-Qwon Oh, \emph{Symplectic ideals of Poisson algebras and the Poisson
  structure associated to quantum matrices}, Comm. Algebra \textbf{27} (1999),
  2163--2180.

\bibitem{Oh10}
Sei-Qwon Oh, \emph{Poisson prime ideals of Poisson polynomial rings}, Comm. Algebra
  \textbf{35} (2007), 3007--3012.

\bibitem{Oh7}
Sei-Qwon Oh, \emph{Quantum and Poisson structures of multi-parameter symplectic and
  Euclidean spaces}, J. Algebra \textbf{319} (2008), 4485--4535.

\bibitem{panov} A. N. Panov, \emph{$n$-Poisson and $n$-Sklyanin brackets}, Journal of Mathematical Sciences, {\bf 110}, (2002) 2322--2329.

\bibitem{Przybysz} R.Przybysz, \emph{On one class of exact Poisson structures},
J. Math. Phys. {\bf 42}, (2001) 1913--1920.


\bibitem{commview}
L. H. Rowen, Graduate Algebra: Commutative View, American Mathematical Society, Providence RI, 2006.


\bibitem{Sha}
R.~Y. Sharp, {Steps in commutative algebra}, Second edition, London Math.
  Soc., Student texts, vol.~51, Cambridge University Press, 2000.

\end{thebibliography}


\end{document}